\newif\ifsiam     
\newif\ifnummat   
\newif\ifcvs      
\newif\ifmoc      

\siamfalse
\nummatfalse
\cvsfalse
\mocfalse
 
\nummattrue

\ifnummat
    \RequirePackage{fix-cm}
    \documentclass[smallextended]{svjour3}
    \smartqed
\fi

\ifmoc
    \documentclass{amsart}
    \newtheorem{theorem}{Theorem}[section]
    \newtheorem{lemma}[theorem]{Lemma}

    \theoremstyle{definition}
    \newtheorem{definition}[theorem]{Definition}

    \theoremstyle{remark}
    \newtheorem{remark}[theorem]{Remark}

    \numberwithin{equation}{section}

\fi

\usepackage{bm,amssymb,amsmath,stmaryrd}
\usepackage{esint}

\ifsiam
    \usepackage[all]{xy}
    \usepackage{braket,amsfonts,amsopn}
    \usepackage{makecell,multirow,diagbox}
    \usepackage{mathrsfs}
    \usepackage{graphicx,epstopdf}
    \usepackage{subfigure}
    \numberwithin{equation}{section}
    \numberwithin{theorem}{section}
\else
    \usepackage[all]{xy}
    \usepackage[colorlinks]{hyperref}
    \usepackage{multirow,subfigure}
    \usepackage{graphicx,epstopdf}
    \numberwithin{lemma}{section}
    \numberwithin{theorem}{section}
    \numberwithin{definition}{section}
    \numberwithin{corollary}{section}
    \numberwithin{equation}{section}
\fi

\DeclareMathOperator{\divg}{div}

\newcommand{\ab}[2]{\langle#1,#2\rangle}

\newcommand{\sumth}[1]{\sum_{K\in\mathcal{T}_{h}}#1}
\newcommand{\Th}{\mathcal{T}_{h}}

\newcommand{\Eh}{\mathcal{E}_{h}}

\newcommand{\vertiii}[1]{{\left\vert\kern-0.25ex\left\vert\kern-0.25ex\left\vert #1 
    \right\vert\kern-0.25ex\right\vert\kern-0.25ex\right\vert}}

\def\yulAuthor{Yuwen Li}

\ifsiam
    
    \def\yulShortAuthor{Y.~Li}
\else
    
    \def\yulShortAuthor{Y. Li}
\fi

\def\yulAddress{Department of Mathematics, The Pennsylvania State University,
University Park, PA 16802.}

\def\yulEmail{yuwenli925@gmail.com}

\title{Superconvergent flux recovery of the Rannacher-Turek nonconforming element}

\def\shortTitle{Superconvergence of the Rannacher-Turek element}

\def\myKeywords{superconvergence, rectangular meshes,  Rannacher--Turek element, Raviart--Thomas element, Crouzeix--Raviart element}

\def\myAMS{65N15, 65N30}

\def\myAbstract{
This work presents superconvergence estimates of the nonconforming Rannacher--Turek element for second order elliptic equations on any cubical meshes in $\mathbb{R}^{2}$ and $\mathbb{R}^{3}$. In particular, a corrected numerical flux is shown to be superclose to the Raviart--Thomas interpolant of the exact flux. We then design a superconvergent recovery operator based on local weighted averaging. Combining the supercloseness and the recovery operator, we prove that the recovered flux superconverges to the exact flux. As a by-product, we obtain a superconvergent recovery estimate of the Crouzeix--Raviart element method for general elliptic equations.
}

\begin{document}


\ifnummat
   \author{\yulAuthor%
           }
  \institute{\yulShortAuthor : \yulAddress\, {Email:}\yulEmail
              }
  \date{Received:  \  / Accepted: date}
  \maketitle
  \begin{abstract}\myAbstract\end{abstract}
  \begin{keywords}\myKeywords\end{keywords}
  \begin{subclass}\myAMS \end{subclass}
  \markboth{\yulShortAuthor }{\shortTitle}
\fi

\ifmoc
    \bibliographystyle{amsplain}
    \author[\yulShortAuthor]{\yulAuthor}
    \address{\yulAddress}
    \email{\yulEmail}

    \subjclass[2010]{Primary \myAMS}
    \date{\today}
    \begin{abstract}\myAbstract\end{abstract}
    \maketitle
    \markboth{
    \yulShortAuthor }{\shortTitle}
\fi

\section{Introduction and preliminaries}\label{sec1}
Finite element superconvergent recovery is quite popular in practice for its simplicity and ability to develop asymptotically exact a posteriori error estimators. The theory of superconvergent recovery for conforming Lagrange elements is well-established, see, e.g., \cite{BS1977,Thomee1977,ZZ1987,ZZ1992,BX2003a,BX2003b,BaXuZheng2007,XZ2003,ZhangNaga2005}. Let $u_{h}$ be the finite element solution approximating the PDE solution $u$. The framework of superconvergent recovery is often divided into two steps. The starting point is a supercloseness estimate between $u_{h}$ and the finite element \emph{canonical interpolant} $u_{I}$, where $u_I$ and $u$ share the same degrees of freedom (dofs) corresponding to certain finite element. Then a postprocessed solution $R_{h}u_{h}$ is shown to superconverge to $u$ in suitable norm, provided $R_{h}$ is a bounded operator with super-approximation property. 

On the other hand, since the interelement boundary continuity of nonconforming elements is very weak, superconvergence analysis of nonconforming methods is often more difficult and limited. The  Crouzeix--Raviart (CR) \cite{CR1973,BS2008} element for the Poisson equation is an important model problem for the analysis of nonconforming methods. In this case, it can be numerically observed that the CR canonical interpolant $u_{I}$ and the finite element solution $u_{h}$ are not superclose in the energy norm. Hence the aforementioned recovery framework does not work. In \cite{Ye2001}, Ye developed superconvergence estimates of the CR element using least-squares surface fitting \cite{Wang2000,WangYe2001}. Guo and Huang \cite{GuoZhang2015} presented a polynomial preserving gradient recovery method for the CR element with numerically confirmed superconvergence. Based on an equivalence between the CR method and the lowest order Raviart--Thomas (RT) method for Poisson's equation (cf.~\cite{Marini1985,AB1985}), Hu and Ma \cite{HM2016} proved a recovery-type superconvergence estimate for the CR element using superconvergence of RT elements in \cite{Brandts1994}. This result is then improved and generalized in e.g., \cite{YL2018,HuMa2018,ZHY2019}.  Readers are also referred to e.g., \cite{ChenLi1994,Chen2002,MS2009,LN2008} and references therein for superconvergence of other nonconforming elements. 

The nonconforming Rannacher--Turek (NCRT) element \cite{RT1992} is a natural generalization of the CR element on quadrilateral meshes. It is noted that there is a superconvergence estimate of the NCRT element at some special points under certain mildly distorted \emph{square} meshes, see \cite{MSX2006}. For the Poisson equation, it has been shown in \cite{LTZ2005} that several rectangular nonconforming methods do not admit natural supercloseness estimates. In particular, $u_{I}$ and $u_{h}$ from the NCRT element are superclose in the energy norm only under \emph{square} meshes. To overcome this barrier, the authors of \cite{LTZ2005} enriched the NCRT element by one degree of freedom at the centroid of each element and proved superconvergent gradient recovery estimates of the modified nonconforming element. 

In this paper, we shall consider the standard NCRT method \eqref{RT} for solving the general elliptic equation \eqref{elliptic}.  First we compute a corrected numerical flux $\bm{\sigma}_{h}$ from the NCRT finite element solution, see Theorem \ref{superclose}. We shall show that $\bm{\sigma}_{h}$ is superclose to $\Pi_{h}(a\nabla u)$ by comparing it with an auxiliary $H(\divg)$-conforming flux $\bar{\bm{\sigma}}_{h}$ and using well-established superconvergence tools and techniques for RT elements in e.g., \cite{Duran1990,Brandts1994,YL2018}. Here $\Pi_{h}$ is the canonical interpolation of the lowest order rectangular RT element. We then construct a local edge-based weighted averaging operator $A_{h}$, which makes $\|a\nabla u-A_{h}\Pi_{h}(a\nabla u)\|$ supersmall on any rectangular mesh. Hence $A_{h}\bm{\sigma}_{h}$ superconverges to $a\nabla u$ on any rectangular mesh by a triangle-inequality argument. To the best of our knowledge, this is the first superconvergent recovery method for the NCRT element on arbitrary rectangular meshes. As far as we know, there is no superconvergence analysis of the tetrahedral CR element in $\mathbb{R}^3.$ In contrast, our superconvergence results could be directly generalized to the  cubic NCRT element in $\mathbb{R}^3$, see Section \ref{sec4}.

For elliptic equations with variable coefficients and lower order terms, Arbogast and Chen in \cite{AC1995} can reformulate various mixed methods as modified nonconforming methods. However, the general equivalence expression is complicated and it is unclear how far the standard nonconforming finite element solution is from the modified one. On the other hand, superconvergence analysis of $H(\divg)$-conforming mixed finite elements is well established, see, e.g., \cite{Duran1990,Brandts1994,YL2018,BaLi2019}. Hence we shall relate nonconforming methods to their mixed counterparts as in \cite{HM2016}. In our superconvergence analysis, it is not necessary to rewrite the NCRT method \eqref{RT} as an equivalent mixed method for the \emph{general elliptic equation}. All we need is the equivalence given by Lemma \ref{mainlemma} for the \emph{Poisson} equation. As far as we know, it is the first superconvergence estimate of the CR and NCRT element methods for the general elliptic equation.

In the rest of this section, we introduce preliminary definitions and notations. Let $\Omega=[\omega_1,\omega_2]\times[\omega_3,\omega_4]\subset\mathbb{R}^{2}$ be a rectangle. Consider the second order elliptic equation 
\begin{subequations}\label{elliptic}
\begin{align}
-\nabla\cdot(a\nabla u)+\bm{b}\cdot\nabla u+cu&=f\quad \text{in}\ \Omega,\\
u&=g\quad \text{on}\ \partial\Omega,
\end{align}
\end{subequations} 
where $a(\bm{x})\geq a_{0}>0$ for all $\bm{x}=(x_{1},x_{2})^{T}\in\Omega$, and $a, \bm{b}, c$, and $f$ are smooth functions in $\bm{x}$ on $\bar{\Omega}$.

Let $\Th$ be a partition of $\Omega$ by rectangles. Given a rectangle $K\in\Th$, let $\ell_{K,1}$ and $\ell_{K,2}$ denote the width and height of $K$ and $h=\max_{K\in\Th}\max(\ell_{K,1},\ell_{K,2})$ the mesh size. We assume that $h<1$ and $\Th$ is nondegenerate, i.e. 
$$\max_{K\in\Th}\max\left\{\frac{\ell_{K,1}}{\ell_{K,2}},\frac{\ell_{K,2}}{\ell_{K,1}}\right\}\leq C_{\Th}<\infty,$$ where $C_{\Th}$ is an absolute constant independent of $h$. Let $\Eh$, $\Eh^{o}$, and $\Eh^{\partial}$ denote the set of edges, interior edges, and boundary edges, respectively. The following edge-based patch $\omega_{E}$ will be frequently used. 
\begin{enumerate}
\item For $E\in\Eh^{o}$, let $\omega_{E}=K^{+}\cup K^{-}$ where $K^{+}$ and $K^{-}$ are the two adjacent rectangles sharing $E$. 
\item For $E\in\Eh^{\partial}$, let $\omega_{E}=K$, where $K$ is the rectangle having $E$ as an edge.
\end{enumerate} 
The NCRT finite element space is defined as
\begin{equation*}
\begin{aligned}
\mathcal{V}_{g,h}:=&\{v_{h}\in L^{2}(\Omega): v_{h}|_{K}\in\text{span}\{1,x_{1},x_{2},x_{1}^{2}-x_{2}^{2}\}\text{ for\ all}\ K\in\Th, \\
&\fint_{E}v_{h}\text{ is\ single-valued for all}\ E\in\Eh^{o}, \fint_{E}v_{h}=\fint_{E}g\text{ for all } E\in\Eh^{\partial}\},
\end{aligned}
\end{equation*}
where $\fint_{E}v:=\frac{1}{|E|}\int_{E}v$ is the mean value of $v$ on $E$.  The name `nonconforming' is due to the fact $\mathcal{V}_{g,h}\not\subseteq H^1(\Omega)$. Let  $$H^1(\Th):=\{v\in L_2(\Omega): v|_K\in H^1(K)~\forall K\in\Th\}$$  be the space of piecewise $H^1$ functions and $\nabla_{h}$ denote the piecewise gradient w.r.t.~$\Th$, namely, 
$$(\nabla_h v)|_K:=\nabla (v|_K),\quad\forall v\in H^1(\Th),\quad \forall K\in\Th.$$
The NCRT method for \eqref{elliptic} is to find $u_{h}\in\mathcal{V}_{g,h}$, such that
\begin{equation}\label{RT}
\ab{a\nabla_{h}u_{h}}{\nabla_{h}v}+\ab{\bm{b}\cdot\nabla_{h}u_{h}}{v}+\langle cu_{h},v\rangle=\ab{f}{v},\quad\forall v\in\mathcal{V}_{0,h},
\end{equation}
where $\ab{\cdot}{\cdot}$ is the $L_2(\Omega)$-inner product.
Throughout this paper, we adopt the notation $A\lesssim B$ when $A\leq CB$ for some generic constant $C$ that is independent of $h$. We assume that the standard a priori error estimate for the NCRT method holds:
\begin{equation}\label{apriori}
\|u-u_{h}\|+h\|\nabla_{h}(u-u_{h})\|\lesssim h^{2}\|u\|_{H^{2}},
\end{equation}
where $\|\cdot\|$ denotes the norm $\|\cdot\|_{L_2(\Omega)}$ and $\|\cdot\|_{H^{2}}$ abbreviates $\|\cdot\|_{H^{2}(\Omega)}$, similar for other Sobolev norms. Readers are referred to \cite{BS2008} for the analogue of \eqref{apriori} for the CR method. The estimate \eqref{apriori} implies that \eqref{RT} is a first order method in the discrete energy norm $\|\nabla_h\cdot\|$. Therefore, an improved recovery-type error estimate of order $1+s$ suffices to declare superconvergence,  where $s>0$ is an absolute constant. Similarly, we say two functions are superclose whenever the $\|\nabla_h\cdot\|$-distance between them is $O(h^{1+s}).$

The following NCRT element space $\widetilde{\mathcal{V}}_{h}$ using DOFs based on pointwise function evaluation will be used in Section \ref{sec3}.
\begin{equation*}
\begin{aligned}
\widetilde{\mathcal{V}}_{h}:=&\{v_{h}\in L^{2}(\Omega): v_{h}|_{K}\in\text{span}\{1,x_{1},x_{2},x_{1}^{2}-x_{2}^{2}\}\text{ for\ all}\ K\in\Th, \\
&v_{h}\text{ is\ continuous at the midpoint of each}\ E\in\Eh^{o}\}.
\end{aligned}
\end{equation*}
Let $Q_{k,l}(K)$ denote the set of polynomials of degree $\leq k$ in $x_{1}$ and of degree $\leq l$ in $x_{2}$ on the element $K$. Let $$H(\divg,\Omega):=\{\bm{\tau}\in L_2(\Omega)\times L_2(\Omega): \nabla\cdot\bm{\tau}\in L_2(\Omega)\}.$$ The lowest order rectangular RT finite element space is
\begin{equation*}
\mathcal{RT}_{h}:=\{\bm{\tau}_{h}\in H(\divg,\Omega): \bm{\tau}_{h}|_{K}\in Q_{1,0}(K)\times Q_{0,1}(K)\text{ for all }K\in\Th\}.
\end{equation*}
For convenience we also introduce the broken RT space
$$\mathcal{RT}_{h}^{-1}:=\{\bm{\tau}_{h}\in L_2(\Omega)\times L_2(\Omega): \bm{\tau}_{h}|_{K}\in Q_{1,0}(K)\times Q_{0,1}(K), \forall K\in\Th\}.$$ 
The dofs for $\mathcal{RT}_h$ consist of integrals of normal components of a vector-valued function on each edge in $\Th$. Given $\bm{\tau}\in H^1(\Omega)\times H^1(\Omega)$, the RT canonical interpolant $\Pi_{h}\bm{\tau}$ is the unique finite element function in $\mathcal{RT}_{h}$ such that
\begin{equation}\label{RTinterpolation}
    \int_{E}(\Pi_{h}\bm{\tau})\cdot\bm{n}_E=\int_{E}\bm{\tau}\cdot\bm{n}_E,\quad\forall E\in\Eh,
\end{equation} 
where $\bm{n}_E$ is a unit normal to $E$. Let $P_{h}$ be the $L_2(\Omega)$-projection onto the space of piecewise constant functions. It is well known that 
\begin{equation}\label{commuting}
\nabla\cdot\Pi_{h}\bm{\tau}=P_{h}\nabla\cdot\bm{\tau}.
\end{equation}
Let $E\in\Eh^{o}$ and $K^{+}, K^{-}$ be the two rectangles sharing $E$. Let $\bm{n}^{+}$ and $\bm{n}^{-}$ denote the outward unit normal induced by $K^{+}$ and $K^{-}$ respectively. In the analysis of nonconforming methods, it is convenient to introduce notations for jumps and averages on $E$: 
\begin{equation*}
\begin{aligned}
&\llbracket\bm{\tau}\rrbracket:=\bm{\tau}|_{K^{+}}\cdot\bm{n}^{+}+\bm{\tau}|_{K^{-}}\cdot\bm{n}^{-},\\
&\{\bm{\tau}\}:=(\bm{\tau}|_{K^{+}}+\bm{\tau}|_{K^{-}})/2,\\
&\llbracket v\rrbracket:=(v|_{K^{+}}\bm{n}^{+}+v|_{K^{-}}\bm{n}^{-})/2,\\
&\{v\}:=(v|_{K^{+}}+v|_{K^{-}})/2,
\end{aligned}
\end{equation*}
where $\bm{\tau}$ is a vector and $v$ is a scalar. For $E\in\Eh^{\partial}$,
\begin{equation*}
\llbracket\bm{\tau}\rrbracket:=\bm{\tau}\cdot\bm{n},\quad\{v\}:=v,\quad\llbracket v\rrbracket:=\bm{0}.
\end{equation*}
where $\bm{n}$ is the outward unit normal to $\partial\Omega$. It is readily checked that
\begin{equation}\label{product}
\begin{aligned}
\llbracket\bm{\tau}v\rrbracket&=\llbracket\bm{\tau}\rrbracket\{v\}+\llbracket v\rrbracket\cdot\{\bm{\tau}\}.
\end{aligned}
\end{equation}
By these notations, a useful fact is that 
\begin{equation}\label{equivRT}
\bm{\tau}_{h}\in\mathcal{RT}_{h}\text{ if\ and\ only\ if } \bm{\tau}_{h}\in\mathcal{RT}_{h}^{-1}\text{ and }\llbracket\bm{\tau}_{h}\rrbracket=0~\forall E\in\Eh^{o}.
\end{equation} 

\textbf{Abbreviation.} For the reader's convenience, abbreviations of finite elements in this paper are summarized as follows.
\begin{align*}
    &\text{Rannacher--Turek: NCRT}\\
    &\text{Raviart--Thomas: RT}\\
    &\text{Crouzeix--Raviart: CR}
\end{align*} 

The rest of this paper is organized as follows. Section \ref{sec2} discusses the supercloseness estimate in Theorem \ref{superclose}. In Section \ref{sec3}, we propose a postprocessing operator and prove the recovery superconvergence estimate in Theorem \ref{superconvergence}. In Section \ref{sec4}, we extend our superconvergence analysis to the CR element and NCRT element in $\mathbb{R}^3$. Numerical experiments are presented in Section \ref{sec5}. Concluding remarks are given in Section \ref{sec6}.

\section{Supercloseness}\label{sec2}
In this section, we derive a supercloseness estimate for the NCRT element, which is essential to develop superconvergent flux recovery.
First we need a lemma in the spirit of Marini (cf. \cite{Marini1985}). 
\begin{lemma}\label{mainlemma}
Let $\bar{f}$ be a piecewise constant, $\bm{\tau}_{h}|_{K}\in Q_{1,0}(K)\times Q_{0,1}(K)$ and $\nabla\cdot(\bm{\tau}_{h}|_{K})=0$ for all $K\in\Th$. Assume that
\begin{equation}\label{var}
\ab{\bm{\tau}_{h}}{\nabla_{h}v}=\ab{\bar{f}}{v}
\end{equation}
for all $v\in\mathcal{V}_{0,h}$. Then $\bm{\tau}_{h}-\bar{f}\bm{r}_{h}\in\mathcal{RT}_{h},$ with
$$\bm{r}_{h}|_{K}(x_1,x_2):=\left( \frac{\ell_{K,2}^{2}}{\ell_{K,1}^{2}+\ell_{K,2}^{2}}(x_{1}-x_{K,1}),~\frac{\ell_{K,1}^{2}}{\ell_{K,1}^{2}+\ell_{K,2}^{2}}(x_{2}-x_{K,2}) \right)^{T},$$
where $K=[x_{1,i},x_{1,i+1}]\times[x_{2,j},x_{2,j+1}]$, $\ell_{K,1}=x_{1,i+1}-x_{1,i}$, $\ell_{K,2}=x_{2,j+1}-x_{2,j}$, and $(x_{K,1},x_{K,2})^{T}$ is the centroid of $K$.
\end{lemma}
\begin{proof}
Consider any vertical edge $E\in\Eh^{o}$ and the two rectangles 
$$K^{-}=[x_{1,i},x_{1,i+1}]\times[x_{2,j},x_{2,j+1}],\quad K^{+}=[x_{1,i+1},x_{1,i+2}]\times[x_{2,j},x_{2,j+1}]$$ 
sharing it. Let $v\in\mathcal{V}_{0,h}$ be the basis function such that 
$$\fint_{E}v_{E}=1,\quad\fint_{E'}v_{E}=0\text{ for } \Eh\ni E'\neq E.$$ 
Note that $\bm{\tau}_{h}\cdot(1,0)^{T}$ is a constant on $E$. It then follows from \eqref{var} with $v=v_E$, $\nabla_{h}\cdot\bm{\tau}_{h}=0$ and integration by parts that
\begin{equation}\label{inter}
\int_{E}\llbracket\bm{\tau}_{h}\rrbracket=\int_{K^{+}\cup K^{-}}\bar{f}v_{E}.
\end{equation}
Direct calculation shows that
\begin{equation}\label{intvE}
    \int_{K^{\pm}}v_{E}=\frac{|K^{\pm}|\ell_{K^{\pm},2}^{2}}{2(\ell_{K^{\pm},1}^{2}+\ell_{K^{\pm},2}^{2})}.
\end{equation}
Then combining \eqref{intvE} with \eqref{inter} and the definition of $\bm{r}_h$ yields
\begin{equation}\label{edgejump}
\llbracket\bm{\tau}_{h}-\bar{f}\bm{r}_{h}\rrbracket=0\text{ on }E.
\end{equation}
Similarly, \eqref{edgejump} also holds for horizontal edges. Combining \eqref{edgejump} with the fact $(\bm{\tau}_{h}-\bar{f}\bm{r}_{h})|_{K}\in Q_{1,0}(K)\times Q_{0,1}(K)$, we conclude that $\bm{\tau}_{h}-\bar{f}\bm{r}_{h}\in\mathcal{RT}_{h}.$
\qed\end{proof}
\begin{remark}
It seems that the NCRT method using dofs based on pointwise function evaluation does not have a similar equivalence. 
\end{remark}
To apply Lemma \ref{mainlemma}, we then introduce the auxiliary nonconforming method: Find $\bar{u}_{h}\in\mathcal{V}_{g,h},$ such that
\begin{equation}\label{auxnc}
\ab{a\nabla_{h}\bar{u}_{h}}{\nabla_{h}v}=\ab{P_{h}(f-cu-\bm{b}\cdot\nabla u)}{v},\quad\forall v\in\mathcal{V}_{0,h}.
\end{equation}
The following lemma shows that $u_{h}$ and $\bar{u}_{h}$ are superclose in the $H^{1}$-norm.
\begin{lemma}\label{superuubar}
Let $u_{h}$ and $\bar{u}_{h}$ solve \eqref{RT} and \eqref{auxnc}, respectively. Then
$$\|\nabla_{h}(u_{h}-\bar{u}_{h})\|\lesssim h^{2}\|u\|_{H^{2}}.$$
\end{lemma}
\begin{proof}
Subtracting \eqref{auxnc} from \eqref{RT} gives
 $$\ab{a\nabla_{h}(u_{h}-\bar{u}_{h})}{\nabla_{h}v}\\
=\ab{f-cu_{h}-\bm{b}\cdot\nabla_{h}u_{h}-P_{h}(f-cu-\bm{b}\cdot\nabla u)}{v},$$
where $v\in\mathcal{V}_{0,h}$.
It then follows from \eqref{apriori} 
that  
\begin{equation}\label{all}
\begin{aligned}
&\ab{a\nabla_{h}(u_{h}-\bar{u}_{h})}{\nabla_{h}v}\\
&=\ab{f-cu-\bm{b}\cdot\nabla u-P_{h}(f-cu-\bm{b}\cdot\nabla u)}{v-P_{h}v}\\
&\quad+\ab{c(u-u_{h})}{v}+\ab{\bm{b}\cdot\nabla_{h}(u-u_{h})}{v}\\
&=O(h^{2})(\|f\|_{H^{1}}+\|u\|_{H^{2}})\|\nabla_{h}v\|+\ab{\bm{b}\cdot\nabla_{h}(u-u_{h})}{v}.
\end{aligned}
\end{equation}
It remains to show that $\ab{\bm{b}\cdot\nabla_{h}(u-u_{h})}{v}$ is supersmall. By integrating by parts, \eqref{product}, and $\fint_{E}\llbracket u-u_{h}\rrbracket=\bm{0}$, we have
\begin{equation*}
\begin{aligned}
&\ab{\bm{b}\cdot\nabla_{h}(u-u_{h})}{v}\\
&\quad=\sumth\int_{\partial K}(u-u_{h})v\bm{b}\cdot\bm{n}-\int_{K}(u-u_{h})\nabla\cdot(\bm{b}v)\\
&\quad=\sum_{E\in\Eh}\int_{E}\{u-u_{h}\}\llbracket v\bm{b}-\bm{c}_{E}\rrbracket+\llbracket u-u_{h}\rrbracket\cdot\{v\bm{b}-\bm{d}_{E}\}\\
&\qquad-\int_{\Omega}(u-u_{h})\nabla_{h}\cdot(\bm{b}v)
\end{aligned}
\end{equation*}
for any constants $\bm{c}_{E}\in\mathbb{R}^{2}$ and $\bm{d}_{E}\in\mathbb{R}^{2}$. In particular, let $\bm{c}_{E}=\bm{d}_{E}=\bm{b}(m_{E})\fint_{E}v$, where $m_E$ is the midpoint of $E.$ By the trace inequality
\begin{equation}\label{trace}
\|w\|_{L_2(\partial K)}\lesssim h^{-\frac{1}{2}}\|w\|_{L_2(K)}+h^{\frac{1}{2}}\|\nabla w\|_{L_2(K)},
\end{equation}
we have 
\begin{equation}\label{trace1}
\begin{aligned}
&\|\{u-u_{h}\}\|_{L_2(E)}+\|\llbracket u-u_{h}\rrbracket\|_{L_2(E)}\\
&\quad\lesssim h^{-\frac{1}{2}}\|u-u_{h}\|_{L_2(\omega_{E})}+h^{\frac{1}{2}}\|\nabla_{h}(u-u_{h})\|_{L_2(\omega_{E})}
\end{aligned}
\end{equation}
and 
\begin{equation}\label{trace2}
\|\llbracket v\bm{b}-\bm{c}_{E}\rrbracket\|_{L_2(E)}+\|\{v\bm{b}-\bm{d}_{E}\}\|_{L_2(E)}\\
\lesssim h^{\frac{1}{2}}\|\nabla_{h}(\bm{b}v)\|_{L_2(\omega_{E})}.
\end{equation}
It follows from the Cauchy--Schwarz inequality,  \eqref{trace1}, \eqref{trace2} and \eqref{apriori} that 
\begin{equation}\label{buuhv}
\begin{aligned}
&|\ab{\bm{b}\cdot\nabla_{h}(u-u_{h})}{v}|\\
&\lesssim\sum_{E\in\Eh}\big(\|\{u-u_{h}\}\|_{L_2(E)}\|\llbracket v\bm{b}-\bm{c}_{E}\rrbracket\|_{L_2(E)}\\
&\quad+\|\llbracket u-u_{h}\rrbracket\|_{L_2(E)}\|\{v\bm{b}-\bm{d}_{E}\}\|_{L_2(E)}\big)+\|u-u_{h}\|\|\nabla_{h}\cdot(\bm{b}v)\|\\
&\leq\sum_{E\in\Eh}\big(\|u-u_{h}\|_{L_2(\omega_{E})}+h\|\nabla_{h}(u-u_{h})\|_{L_2(\omega_{E})}\big)\|\nabla_{h}(\bm{b}v)\|_{L_2(\omega_{E})}\\
&\quad+\|u-u_{h}\|\|\nabla_{h}\cdot(\bm{b}v)\|\\
&\lesssim\big(\|u-u_{h}\|+h\|\nabla_{h}(u-u_{h})\|\big)\|\nabla_{h}(\bm{b}v)\|+\|u-u_{h}\|\|\nabla_{h}\cdot(\bm{b}v)\|\\
&\lesssim h^{2}\|u\|_{H^{2}}\big(\|v\|+\|\nabla_{h} v\|\big).
\end{aligned}
\end{equation}
Combining \eqref{buuhv} with \eqref{all} and using the discrete Poincar\'e  inequality (cf.~Theorem 10.6.12.~in \cite{BS2008})
$\|v\|\lesssim\|\nabla_{h}v\|$,
we complete the proof.
\qed\end{proof}

Now we are in a position to present supercloseness results. Let $Q_{h}$ be the $L_{2}$-projection onto $\nabla _{h}\mathcal{V}_{0,h}$ and  
$$\bm{\sigma}_{h}:=Q_{h}(a\nabla_{h}u_{h})-\bm{r}_{h}P_{h}(f-cu_{h}-\bm{b}\cdot\nabla_{h}u_{h})$$ 
be the corrected flux, where $\bm{r}_{h}$ is defined in Lemma \ref{mainlemma}. Note that $Q_{h}$ is indeed {\color{red}an} element-by-element projection and $Q_{h}(a\nabla_{h}u_{h})=a\nabla_{h}u_{h}$ if $a$ is a piecewise constant. The next theorem shows that $\bm{\sigma}_h$ approximates the exact flux $\bm{\sigma}:=a\nabla u$ very well. 
\begin{theorem}\label{superclose} 
It holds that
\begin{equation*}
\|\Pi_{h}\bm{\sigma}-\bm{\sigma}_{h}\|\lesssim
h^{2}\|u\|_{H^{3}}.
\end{equation*}
\end{theorem}
\begin{proof}
Let $\bar{\bm{\sigma}}_{h}:=Q_{h}(a\nabla_{h}\bar{u}_{h})-\bm{r}_{h}P_{h}(f-cu-\bm{b}\cdot\nabla u)$. Using the definition of $\bar{u}_h$, $\nabla_{h}\cdot Q_{h}=0$ and Lemma \ref{mainlemma}, we conclude that $\bar{\bm{\sigma}}_{h}\in \mathcal{RT}_{h}\subset H(\divg,\Omega)$. Let $\bm{\tau}_{h}=\Pi_{h}\bm{\sigma}-\bar{\bm{\sigma}}_{h}$. It follows from \eqref{commuting} and $\nabla_{h}\cdot\bm{r}_{h}=1$ that
\begin{equation*}
\nabla\cdot\bm{\tau}_{h}=P_{h}\nabla\cdot(a\nabla u)-P_{h}(f-cu-\bm{b}\cdot\nabla u)=0.
\end{equation*}
Hence $\bm{\tau}_{h}|_{K}=(c_{1}x_{1}+c_{2},-c_{1}x_{2}+c_{3})^{T}$ for some $c_{i}\in\mathbb{R}$ on an element $K\in\Th$. On the other hand, direct calculation shows that
\begin{equation*}
\begin{aligned}
&\int_{K}\bm{r}_{h}\cdot\bm{\tau}_{h}=\int_{K}\bm{r}_{h}\cdot\big(\bm{\tau}_{h}-(c_{2}+c_{1}x_{K,1},c_{3}-c_{1}x_{K,2})^{T}\big)\\
&\quad=\frac{c_{1}}{\ell_{K,1}^{2}+\ell_{K,2}^{2}}\int_{K}\ell_{K,2}^{2}(x_{1}-x_{K,1})^{2}-\ell_{K,1}^{2}(x_{2}-x_{K,2})^{2}=0.
\end{aligned}
\end{equation*}
With the above identity, $\bm{\sigma}=a\nabla u$ and $\bm{\tau}_h\in\nabla_h\mathcal{V}_{0,h},$ we obtain
\begin{equation}\label{total}
\begin{aligned}
&\|\Pi_{h}\bm{\sigma}-\bar{\bm{\sigma}}_{h}\|^{2}=I+II,
\end{aligned}
\end{equation}
where 
\begin{align*}
I=\ab{\Pi_{h}\bm{\sigma}-\bm{\sigma}}{\bm{\tau}_{h}},\quad II=\ab{a\nabla_{h}(u-\bar{u}_{h})}{\bm{\tau}_{h}}.
\end{align*}
By Lemma 3.1 with $k=0$ in \cite{Duran1990} and the Bramble--Hilbert lemma,
\begin{equation}\label{bdI}
|I|\lesssim|\bm{\sigma}|_{H^{2}}\|\bm{\tau}_{h}\|.
\end{equation}
For part $II$, due to $\nabla\cdot(\bm{\tau}_{h}|_{K})=0$, we have
\begin{equation}\label{totalII}
\begin{aligned}
II&=\sum_{K\in\Th}\int_{K}a\nabla(u-\bar{u}_{h})\cdot\bm{\tau}_{h}\\
&=\sum_{K\in\Th}\int_{K}(\nabla(a(u-\bar{u}_{h}))-(u-\bar{u}_{h})\nabla a)\cdot\bm{\tau}_{h}\\
&=II_{1}+II_{2},
\end{aligned}
\end{equation}
where $II_1$ and $II_2$ are given by
\begin{align*}
    II_{1}=\sum_{K\in\Th}\int_{\partial K}a(u-\bar{u}_{h})\bm{\tau}_{h}\cdot\bm{n},\quad II_2=-\ab{(u-\bar{u}_{h})\nabla a}{\bm{\tau}_{h}}.
\end{align*}
The part $II_{2}$ is estimated by Lemma \ref{superuubar} and the a priori estimate \eqref{apriori}:
\begin{equation}\label{bdII2}
|II_{2}|\lesssim h^{2}\|u\|_{H^{2}}\|\bm{\tau}_{h}\|.
\end{equation}
Note that the normal component of $\{\bm{\tau}_{h}\}$ is constant on $E$ and $\llbracket\bm{\tau}_{h}\rrbracket=0$ by \eqref{equivRT}. It then follows from $\fint_{E}\llbracket\bar{u}_{h}\rrbracket=\bm{0}$, \eqref{product} , the trace inequality \eqref{trace}, an inverse inequality, \eqref{apriori}, and Lemma \ref{superuubar}, that
\begin{equation}\label{bdII1}
\begin{aligned}
&II_{1}=\sum_{E\in\Eh}\int_{E}\llbracket a(u-\bar{u}_{h})\bm{\tau}_{h}\rrbracket\\
&\quad=\sum_{E\in\Eh}\int_{E}\llbracket(a-\fint_{E}a)(u-\bar{u}_{h})\rrbracket\cdot\{\bm{\tau}_{h}\}\\
&\quad\lesssim h\sum_{E\in\Eh}\|\llbracket u-\bar{u}_{h}\rrbracket\|_{L^{2}(E)}\|\{\bm{\tau}_{h}\}\|_{L^{2}(E)}\\
&\quad\lesssim h^{\frac{1}{2}}\sum_{E\in\Eh}(h^{-\frac{1}{2}}\|u-\bar{u}_{h}\|_{L^{2}(\omega_{E})}+h^{\frac{1}{2}}\|\nabla_{h}(u-\bar{u}_{h})\|_{L^{2}(\omega_{E})})\|\bm{\tau}_{h}\|_{L^{2}(\omega_{E})}\\
&\quad\lesssim \big(\|u-\bar{u}_{h}\|+h\|\nabla_{h}(u-\bar{u}_{h})\|\big)\|\bm{\tau}_{h}\|\lesssim h^{2}\|u\|_{H^{2}}\|\bm{\tau}_{h}\|.
\end{aligned}
\end{equation}
Combining \eqref{total}--\eqref{bdII1}, we obtain
\begin{equation}\label{finalsigmabar}
\|\Pi_{h}\bm{\sigma}-\bar{\bm{\sigma}}_{h}\|\lesssim
h^{2}\|u\|_{H^{3}}.
\end{equation}
On the other hand, Lemma \ref{superuubar} implies 
\begin{equation}\label{sigmasigmabar}
\|{\bm{\sigma}}_{h}-\bar{\bm{\sigma}}_{h}\|\lesssim h^{2}\|u\|_{H^{2}}.
\end{equation}
The theorem then follows from \eqref{finalsigmabar} and \eqref{sigmasigmabar}.
\qed\end{proof}

Key ingredients in the proof of Theorem \ref{superclose} include the RT flux $\bar{\bm{\sigma}}$ and the superconvergence estimate \eqref{bdI} for rectangular RT elements. Similarly, Cockburn et al.~\cite{CoGuWa2009} postprocessed the approximate fluxes from a large class of discontinuous Galerkin methods to obtain $H(\text{div})$-conforming RT fluxes, which facilitates the superconvergence analysis of recovered potentials. 

Theorem \ref{superclose} shows that the corrected flux $\bm{\sigma}_h$ is superclose to the canonical RT interpolant $\Pi_h\bm{\sigma}.$ In contrast, many supercloseness results in the literature are based on corrected interpolants/projections that are superclose to the numerical solution. Readers are referred to \cite{Chen2002,ChenHu2013,CSYZ2015,CaoHuang2017,CSYZ2018} and references therein for superconvergence analysis of $H^1$-conforming and discontinuous Galerkin methods by corrected projection technique  using orthogonal polynomials.

\section{Postprocessing and superconvergence}\label{sec3} 
For the rectangular RT element, Dur\'an \cite{Duran1990} gave a postprocessing operator $K_{h}^{D}$ satisfying
\begin{subequations}\label{Kh}
\begin{align}
&\|K_{h}^{D}\bm{\tau}_{h}\|\lesssim\|\bm{\tau}_{h}\|\text{ for all }\bm{\tau}_{h}\in\mathcal{RT}_{h},\\
&\|\bm{\sigma}-K_{h}^{D}\Pi_{h}\bm{\sigma}\|\lesssim h^{2}|\bm{\sigma}|_{H^{2}}.
\end{align}
\end{subequations}
Here the input for $K_{h}^{D}$ needs to be $H(\divg)$-conforming. Now assume the corrected flux $\bm{\sigma}_{h}\in\mathcal{RT}_{h}$, e.g., $f$ is piecewise constant, $\bm{b}=\bm{0}$, and $c=0$. Using \eqref{Kh}, Theorem \ref{superclose}, and the triangle inequality
\begin{equation*}
\|a\nabla u-K_{h}^{D}\bm{\sigma}_{h}\|\leq\|a\nabla u-K_{h}^{D}\Pi_{h}\bm{\sigma}\|+\|K_{h}^{D}(\Pi_{h}\bm{\sigma}-\bm{\sigma}_{h})\|,
\end{equation*}
we obtain 
\begin{equation*}
\|a\nabla u-K_{h}^{D}\bm{\sigma}_{h}\|\lesssim h^{2}\|u\|_{H^{3}}.
\end{equation*}

However, $\bm{\sigma}_{h}\in\mathcal{RT}^{-1}_{h}$ and $\bm{\sigma}_{h}\notin\mathcal{RT}_{h}$ in general and thus $K_{h}^{D}$ cannot be directly applied to $\bm{\sigma}_{h}$. In this section, we introduce a simple recovery operator $A_{h}$ by local weighted averaging. 
\begin{definition}\label{defAh}
The operator $A_{h}: \mathcal{RT}_{h}^{-1}\rightarrow\widetilde{\mathcal{V}}_{h}$ is defined as follows. 
\begin{enumerate}
\item For each $E\in\Eh^{o}$, let $m$ be the midpoint of $E$. Let $K^{+}$ and $K^{-}$ be the two rectangles sharing $E$ as an edge. Define
\begin{equation*}
(A_{h}\bm{\tau}_{h})(m):=\frac{|K^{-}|}{|K^{+}|+|K^{-}|}\bm{\tau}_{h}|_{K^{+}}(m)+\frac{|K^{+}|}{|K^{+}|+|K^{-}|}\bm{\tau}_{h}|_{K^{-}}(m).
\end{equation*}
\item For each $E\in\Eh^{\partial}$, let $m$ denote the midpoint of $E$ and $K$ the element having $E$ as an edge. Let $E'$ be the edge of $K$ opposite to $E$  with midpoint $m'$. Let $K'$ be the other element having $E'$ as an edge and $m''$ the midpoint of the edge of $K'$ opposite to $E'$. Define
\begin{equation*}
(A_{h}\bm{\tau}_{h})(m):=((A_{h}\bm{\tau}_{h})(m')-w'(A_{h}\bm{\tau}_{h})(m''))/w,
\end{equation*}
where 
\begin{equation*}
w=\frac{|K'|}{|K|+|K'|},\quad w'=\frac{|K|}{|K|+|K'|}.
\end{equation*}
\end{enumerate}
Then $A_{h}\bm{\tau}_{h}$ is the unique finite element in $\widetilde{\mathcal{V}}_{h}$ whose midpoint values are specified in the above two steps. 
\end{definition}

Note that $A_h\bm{\tau}_h\not\in H^1(\Omega)$ and  the weight constants in Definition \ref{defAh} are not chosen in a standard way. We show that $A_{h}$ has a super-approximation property on any nondegenerate rectangular meshes.
\begin{theorem}\label{superapprox}
For $\bm{\tau}_{h}\in\mathcal{RT}_{h}^{-1}$ and $\bm{\tau}\in H^2(\Omega)$, it holds that
\begin{subequations}
\begin{align}
\|A_{h}\bm{\tau}_{h}\|&\lesssim\|\bm{\tau}_{h}\|,\label{a}\\
\|\bm{\tau}-A_{h}\Pi_{h}\bm{\tau}\|&\lesssim h^{2}|\bm{\tau}|_{H^{2}}.\label{b}
\end{align}
\end{subequations}
\end{theorem}
\begin{proof}
Consider $K\in\Th$ and 
$$\omega_{K}:=\bigcup_{E\subset\partial K}\omega_{E}.$$
Using the stability of $A_{h}$ in the $L_{\infty}$-norm and the inverse inequality, we prove the stability of $A_{h}$ in the $L_{2}$-norm:
\begin{equation*}
\|A_{h}\bm{\tau}_{h}\|_{L_{2}(K)}\lesssim h\|A_{h}\bm{\tau}_{h}\|_{L_{\infty}(K)}\lesssim h\|\bm{\tau}_{h}\|_{L_{\infty}(\omega_{K})}\lesssim \|\bm{\tau}_{h}\|_{L_{2}(\omega_{K})}.
\end{equation*}
\eqref{a} then follows from the above estimate and sum of squares.

Let $E\in\Eh^{o}$ with midpoint $m$ and two adjacent elements $K^{+}, K^{-}$ sharing $E$. For $\bm{\tau}_{1}\in Q_{1,1}(\omega_{E})\times Q_{1,1}(\omega_{E})$, we first want to show $(\bm{\tau}_{1}-A_{h}\Pi_{h}\bm{\tau}_{1})(m)=\bm{0}$. Since $\Pi_{h}$ preserves functions in $Q_{1,0}(\omega_{E})\times Q_{0,1}(\omega_{E})$, it suffices to check when $\bm{\tau}_{1}=(y,0)^{T}$ or $(0,x)^{T}$. By linearity we can assume $m=\bm{0}$ without loss of generality. If $E$ is a horizontal interior edge, let $K^{+}=[-\ell_{1}/2,\ell_{1}/2]\times[0,\ell_{2}^{+}]$, $K^{-}=[-\ell_{1}/2,\ell_{1}/2]\times[-\ell_{2}^{-},0]$. Then,
$$\Pi_{h}\begin{pmatrix}y\\0\end{pmatrix}=\left\{\begin{aligned}(\ell_{2}^{+}/2,0)^{T}\quad\text{if }y>0\\(-\ell_{2}^{-}/2,0)^{T}\quad\text{if }y<0\end{aligned}\right.,\quad\Pi_{h}\begin{pmatrix}0\\x\end{pmatrix}=\begin{pmatrix}0\\0\end{pmatrix}.$$
In each case, $(\bm{\tau}_{1}-A_{h}\Pi_{h}\bm{\tau}_{1})(m)=0$. The same argument works for vertical interior edges.

Let $E\in\Eh^{\partial}$ and $K$ the element having $E$ as an edge. Let $E'$ be the edge of $K$ opposite to $E$ and $K'$ be the element sharing the edge $E'$ with $K$. Let $E''$ be the edge of $K'$ opposite to $E'$ and $K''$ be the element sharing $E''$ with $K'$. Let $\omega_{E}=K\cup K'\cup K''$. By similar argument, we have $(\bm{\tau}_{1}-A_{h}\Pi_{h}\bm{\tau}_{1})(m)=0$ when $\bm{\tau}_{1}\in Q_{1,1}(\omega_{E})\times Q_{1,1}(\omega_{E}).$

Using the property derived in the above three paragraphs, for $\bm{\tau}_{1}\in Q_{1,1}(\omega_{K})\times Q_{1,1}(\omega_{K})$, we have
\begin{equation*}
\begin{aligned}
&\|\bm{\tau}-A_{h}\Pi_{h}\bm{\tau}\|_{L_{2}(K)}\lesssim h\|\bm{\tau}-A_{h}\Pi_{h}\bm{\tau}\|_{L_{\infty}(K)}\\
&\quad\lesssim h\|(\text{id}-A_{h}\Pi_{h})(\bm{\tau}-\bm{\tau}_{1})\|_{L_{\infty}(K)}\lesssim h\|\bm{\tau}-\bm{\tau}_{1}\|_{L_{\infty}(\omega_{K})},
\end{aligned}
\end{equation*}
where $\text{id}$ is the identity mapping. Then by standard finite element approximation theory (cf. Corollary 4.4.7 in \cite{BS2008}), 
\begin{equation}\label{l2max}
\inf_{\bm{\tau}_{1}\in Q_{1,1}(\omega_{K})\times Q_{1,1}(\omega_{K})}\|\bm{\tau}-\bm{\tau}_{1}\|_{L_{\infty}(\omega_{K})}\lesssim h|\bm{\tau}|_{H^{2}(\omega_{K})}
\end{equation}
and thus
\begin{equation}\label{local}
\|\bm{\tau}-A_{h}\Pi_{h}\bm{\tau}\|_{L_2(K)}\lesssim h^{2}|\bm{\tau}|_{H^{2}(\omega_{K})}.
\end{equation}
Then \eqref{b} follows from \eqref{local} and sum of squares.
\qed\end{proof}

Combining Theorems \ref{superclose} and \ref{superapprox}, we obtain the superconvergent flux recovery estimate.
\begin{theorem}\label{superconvergence}
It holds that
\begin{equation*}
\|a\nabla u-A_{h}\bm{\sigma}_{h}\|\lesssim h^{2}\|u\|_{H^{3}}.
\end{equation*}
\end{theorem}
\begin{proof}
Combining Theorems \ref{superclose}, \ref{superapprox} and the triangle inequality
\begin{equation*}
\|a\nabla u-A_{h}\bm{\sigma}_{h}\|\leq\|a\nabla u-A_{h}\Pi_{h}\bm{\sigma}\|+\|A_{h}(\Pi_{h}\bm{\sigma}-\bm{\sigma}_{h})\|
\end{equation*}
completes the proof.
\qed\end{proof}

Consider $\tilde{\bm{\sigma}}_{h}\in\mathcal{RT}_{h}^{-1}$, where
\begin{equation}\label{sigmatilde}
\tilde{\bm{\sigma}}_{h}|_{K}:=Q_{h}(a\nabla_{h}u_{h})-\bm{r}_{h}(f-\bm{b}\cdot\nabla_{h} u_{h}-cu_{h})(\bm{x}_{K}),
\end{equation} 
with $\bm{x}_K=(x_{K,1},x_{K,2})^T$ being the centroid of $K.$
Since $\bm{r}_{h}=O(h)$, we have 
$$\|\tilde{\bm{\sigma}}_{h}-\bm{\sigma}_{h}\|\lesssim h^{2}\|u\|_{H^{2}}.$$
and thus
$$\|a\nabla u-A_{h}\tilde{\bm{\sigma}}_{h}\|\lesssim h^{2}\|u\|_{H^{3}}.$$
$\tilde{\bm{\sigma}}_{h}$ is favorable because of lower computational cost.

\begin{remark}
Let $\widetilde{\mathcal{T}}_h$ be the refinement of $\Th$ by connecting  midpoints of opposite edges of each rectangle in $\Th.$ Let $\phi_h$ be a bilinear nodal basis function on $\widetilde{\mathcal{T}}_h$ scaled and translated such that $\phi_h$ is centered at $\bm{0}$ and $\int_{\mathbb{R}^2}\phi_h=1$. For a uniform $\Th$ and a piecewise constant $\bm{\tau}_h$ on $\Th,$ the convolution $\bm{\tau}_h*\phi_h$ coincides with $A_h\bm{\tau}_h$ at the midpoint of each interior edge in $\Th$. 

Since $\nabla_hu_h$ is not piecewise constant and $\Th$ is not uniform, the edge-based averaging $K_h$ is generally not the same as $\phi_h$-convolution at midpoints of interior edges.
For conforming finite elements, local postprocessing based on spline convolution kernels \cite{BS1977,Thomee1977} are able to produce high order superconvergence on uniform meshes, see also \cite{RyanShu2007} for similar technique in discontinuous Galerkin methods. It would be interesting to check  whether those kernels lead to superconvergence for nonconforming methods.
\end{remark}


\section{Extensions to triangular elements and higher dimensional space}\label{sec4}
In this section, we extend superconvergence analysis in Section \ref{sec3} to triangular CR elements and NCRT elements in $\mathbb{R}^{d}$ with $d\geq3$. 
\subsection{Crouzeix--Raviart elements in $\mathbb{R}^{2}$}
Based on the equivalence between mixed and nonconforming methods for Poisson's equation, a superconvergent recovery for CR elements applied to Poisson's equation has been developed in \cite{HM2016}. We generalize this result for elliptic equations with lower order terms and variable coefficients. In this subsection, let $\Th$ be a triangular mesh on $\Omega$.
The CR finite element space is
\begin{equation*}
\begin{aligned}
\mathcal{V}_{g,h}^\Delta:=&\{v_{h}\in L_2(\Omega): v_{h}|_{K}\in\text{span}\{1,x_{1},x_{2}\}\text{ for\ all}\ K\in\Th, \\
&v_{h}\text{ is\ continuous at the midpoint of each}\ E\in\Eh^{o}, \\
&\fint_{E}v_{h}=\fint_{E}g\text{ for all } E\in\Eh^{\partial}\}.\\
\end{aligned}
\end{equation*}
The CR method for \eqref{elliptic} is to find $u_{h}^\Delta\in\mathcal{V}_{g,h}^\Delta$, such that
\begin{equation*}
\ab{a\nabla_{h}u_{h}^\Delta}{\nabla_{h}v}+\ab{\bm{b}\cdot\nabla_{h}u_{h}^\Delta}{v}+\langle cu_{h}^\Delta,v\rangle=\ab{f}{v},\quad\forall v\in\mathcal{V}_{0,h}^\Delta.
\end{equation*}

The lowest order triangular RT finite element space is
\begin{equation*}
{\mathcal{RT}^\Delta_{h}}:=\{\bm{\tau}_{h}\in H(\divg,\Omega): \bm{\tau}_{h}|_{K}\in\text{span}\left\{\begin{pmatrix}1\\0\end{pmatrix},\begin{pmatrix}0\\1\end{pmatrix},\begin{pmatrix}x_{1}\\x_{2}\end{pmatrix}\right\}\text{ for all }K\in\Th\}.
\end{equation*}
It has been shown in \cite{Marini1985}  that CR and RT finite element spaces are closely related by the following lemma.
\begin{lemma}\label{triMarini}
Let $\bar{f}$ and $\bm{\tau}_{h}$ be piecewise constant functions with respect to $\Th$. Assume that
\begin{equation*}
\ab{\bm{\tau}_{h}}{\nabla_{h}v}=\ab{\bar{f}}{v}
\end{equation*}
for all $v\in\mathcal{V}_{0,h}^\Delta$. Then $\bm{\tau}_{h}-\bar{f}\bm{r}_{h}^\Delta\in{\mathcal{RT}}^\Delta_{h},$ with
$$\bm{r}_{h}^\Delta|_{K}(x_1,x_2):=\frac{1}{2}\left(x_{1}-x_{K,1}, x_{2}-x_{K,2} \right)^{T},$$
where $(x_{K,1},x_{K,2})$ is the centroid of $K$.
\end{lemma}

We say $\Th$ is a uniform parallel mesh if each pair of adjacent triangles in $\Th$ forms a parallelogram. A supercloseness estimate follows from Lemma \ref{triMarini}, a supercloseness estimate for triangular RT elements in \cite{YL2018,HuMa2018}, and the same procedure in Section \ref{sec2}. By abuse of notation, $\Pi_h$ denotes the canonical interpolation onto $\mathcal{RT}_h^\Delta.$
\begin{theorem}\label{supercloseCR}
Let $\Th$ be a uniform parallel mesh. Let $$\bm{\sigma}_{h}^\Delta:=\bar{a}\nabla_{h}u_{h}^\Delta-\bm{r}_{h}^\Delta P_{h}(f-cu_{h}^\Delta-\bm{b}\cdot\nabla_{h}u_{h}^\Delta),$$ 
where $\bar{a}|_{K}=\fint_{K}a$ for $K\in\Th$.
It holds that
\begin{equation*}
\|\Pi_{h}\bm{\sigma}-\bm{\sigma}_{h}^\Delta\|\lesssim
h^{2}|\log h|^{\frac{1}{2}}\|u\|_{W^3_\infty}.
\end{equation*}
\end{theorem}
\begin{proof}
We use similar notations and proceed as in the proof of Theorem \ref{superclose}. Let $\bm{\tau}_{h}=\Pi_{h}\bm{\sigma}-\bar{\bm{\sigma}}_{h}^\Delta$, where $\bar{\bm{\sigma}}_{h}^\Delta=\bar{a}\nabla_{h}\bar{u}_{h}^\Delta-\bm{r}_{h}^\Delta P_{h}(f-cu-\bm{b}\cdot\nabla u)$ and $\bar{u}_h^\Delta$ is the solution to the auxiliary problem \eqref{auxnc} with $\mathcal{V}_{0,h}^\Delta$ replacing $\mathcal{V}_{0,h}$.

It then follows from Lemma \ref{triMarini} that $\bm{\tau}_{h}\in\mathcal{RT}_h^\Delta$ with $\nabla\cdot\bm{\tau}_{h}=0$. Hence $\bm{\tau}_{h}=\nabla^{\perp}w_{h}$ for some continuous piecewise linear function $w_{h}$, where $\nabla^{\perp}=(-\partial_{x_{2}},\partial_{x_{1}})^{T}$. The bound \eqref{bdI} for part $I$ is replaced by 
$$|\ab{\bm{\sigma}-\Pi_{h}\bm{\sigma}}{\nabla^\perp w_{h}}|\lesssim h^{2}|\log h|^{\frac{1}{2}}\|\bm{\sigma}\|_{W^2_\infty}\|\nabla^{\perp}w_{h}\|,$$
which is proved in \cite{YL2018}. The rest of the proof is the same as Theorem \ref{superclose}.
\qed\end{proof}

For the recovery purpose, let \begin{equation*}
\begin{aligned}
\mathcal{V}_{h}^\Delta:=&\{v_{h}\in L_2(\Omega): v_{h}|_{K}\in\text{span}\{1,x_{1},x_{2}\}\text{ for\ all}\ K\in\Th, \\
&v_{h}\text{ is\ continuous at the midpoint of each}\ E\in\Eh^{o}\}.
\end{aligned}
\end{equation*}
Then we consider the postprocessing operator $K_{h}$ defined in \cite{Brandts1994}, see also \cite{Duran1991}. 
\begin{definition}
Let $\bm{\tau}_{h}$ be a piecewise constant function.
\begin{enumerate} 
\item For each $E\in\Eh^{o}$, let $m$ be the midpoint of $E$. Let $K^{+}$ and $K^{-}$ be the two rectangles sharing $E$ as an edge. Define
\begin{equation*}
(K_{h}\bm{\tau}_{h})(m):=\frac{1}{2}\bm{\tau}_{h}|_{K^{+}}(m)+\frac{1}{2}\bm{\tau}_{h}|_{K^{-}}(m).
\end{equation*}
\item For each $E\in\Eh^{\partial}$, let $m$ denote the midpoint of $E$ and $K$ the element having $E$ as an edge. Let $E'$ be another edge of $K$ with midpoint $m'$. Let $K'$ be the other element having $E'$ as an edge and $m''$ the midpoint of the edge of $K'$ that is parallel to $E$. Define
\begin{equation*}
(K_{h}\bm{\tau}_{h})(m):=2(K_{h}\bm{\tau}_{h})(m')-(K_{h}\bm{\tau}_{h})(m'').
\end{equation*}
\end{enumerate}
Then $K_{h}\bm{\tau}_{h}$ is the unique element in $\mathcal{V}_{h}^\Delta$ whose midpoint values are specified in the above two steps. 
\end{definition}

Based on Theorem \ref{supercloseCR}, we obtain the superconvergent recovery for the CR element.
\begin{theorem}
Let $\Th$ be a uniform parallel mesh. Then
\begin{equation*}
\|a\nabla u-K_{h}(\bar{a}\nabla_{h}u_{h}^\Delta)\|\lesssim
h^{2}|\log h|^{\frac{1}{2}}\|u\|_{W^3_\infty}.
\end{equation*}
\end{theorem}
\begin{proof}
The operator $K_{h}$ is known to satisfy Theorem \ref{superapprox} with $K_{h}$ replacing $A_{h}$, see \cite{Brandts1994}.
It then follows from Theorem \ref{supercloseCR} and the same argument in the proof of Theorem \ref{superconvergence} that
\begin{equation}\label{super1}
\|a\nabla u-K_{h}\bm{\sigma}_{h}^\Delta\|\lesssim h^{2}|\log h|^{\frac{1}{2}}\|u\|_{W^3_\infty}.
\end{equation}
Let $p=f-cu-\bm{b}\cdot\nabla u$ and $\tilde{\bm{\sigma}}_{h}^\Delta:=\bar{a}\nabla_{h}u^\Delta_{h}-\bm{r}^\Delta_{h}P_{h}p$. It follows from $\|\bm{r}_h\|_{L_\infty}=O(h)$ and \eqref{apriori} that
\begin{equation}\label{sigmabarCR}
\|\bm{\sigma}_{h}^\Delta-\tilde{\bm{\sigma}}_{h}^\Delta\|\lesssim h^{2}\|u\|_{H^{2}}.
\end{equation}
Let $m$ be the midpoint of any $E\in\Eh^{o}$. We have
\begin{equation*}
\begin{aligned}
&[(K_{h}(\bm{r}_{h}^\Delta P_{h}p)](m)=[K_{h}(\bm{r}_{h}^\Delta p)](m)+[K_{h}(\bm{r}_{h}^\Delta(P_{h}p-p))](m)\\
&\quad=(K_{h}\bm{r}_{h}^\Delta)(m)p(m)+O(h^{2})\|u\|_{W^2_\infty}=O(h^{2})\|u\|_{W^2_\infty}.
\end{aligned}
\end{equation*}
In the last equality, we use  $(K_{h}\bm{r}_{h}^\Delta)(m)=0$.
Similar argument works for $E\in\Eh^{\partial}$.
Hence
\begin{equation}\label{hot}
\|K_{h}(\bm{r}_{h}^\Delta P_{h}p)\|\lesssim\|K_{h}(\bm{r}_{h}^\Delta P_{h}p)\|_{L_{\infty}}\lesssim h^{2}\|u\|_{W^2_\infty}.
\end{equation}
Combining \eqref{super1}-\eqref{hot} and the triangle inequality
\begin{equation*}
\begin{aligned}
&\|a\nabla u-K_{h}(\bar{a}\nabla u_{h}^\Delta)\|\leq\|a\nabla u-K_{h}\bm{\sigma}_{h}^\Delta\|\\
&\quad+\|K_{h}(\bm{\sigma}_{h}^\Delta-\tilde{\bm{\sigma}}_{h}^\Delta)\|+\|K_{h}(\bm{r}_{h}^\Delta P_{h}p)\|
\end{aligned}
\end{equation*}
completes the proof.
\qed\end{proof}

It is noted that $K_{h}$ superconverges on mildly structured meshes, see, e.g., \cite{YL2018}. For superconvergence results on  mildly perturbed uniform  triangular grids, readers are also referred to \cite{LMW2000,BX2003a,XZ2003,BaLi2019,DuZhang2019} and references therein. A disadvantage of $K_{h}$ is that it outputs a nonconforming function which is sometimes undesirable. For a vertex $z$ in $\Th$, let $\omega_{z}$ be the patch which is the union of triangles surrounding $z$. Define
$$\widetilde{K}_{h}(\bar{a}\nabla_{h}u_{h}^\Delta)(z):=\sum_{K\subset\omega_{z}}\frac{|K|}{|\omega_{z}|}\bar{a}\nabla_{h}u_{h}^\Delta|_{K}.$$
We then obtain a nodal averaging procedure $\widetilde{K}_{h}$ and a continuous piecewise linear function $\widetilde{K}_{h}(\bar{a}\nabla_{h}u_{h}^\Delta)$. Following similar argument in this section, it is straightforward to show
$$\|a\nabla u-\widetilde{K}_{h}(\bar{a}\nabla_{h}u_{h}^\Delta)\|\lesssim h^{\frac{3}{2}}\|u\|_{H^{3}},$$
provided $\Th$ is uniformly parallel.

\subsection{Rannacher--Turek elements in $\mathbb{R}^{d}$}
Let $\Omega=\Pi_{j=1}^d[\omega_{j,1},\omega_{j,2}]\subset\mathbb{R}^{d}$ be a hypercube where $d\geq3$ is an integer. We assume that $a, \bm{b}, c, f, g$ in \eqref{elliptic} are functions in $\bm{x}=(x_{1},\ldots,x_{d})^T\in\Omega$. Let $\Th$ be a cubical mesh of $\Omega$, where each element $K$ in $\Th$ is of the form 
$$K=\Pi_{j=1}^d[x_{j,i_j},x_{j,i_j+1}]=[x_{1,i_1},x_{1,i_1+1}]\times[x_{2,i_2},x_{2,i_2+1}]\times\cdots[x_{d,i_d},x_{d,i_d+1}]$$
with $i_1,\ldots,i_d\in\mathbb{Z}^+.$ Let $\mathcal{F}_h$, $\mathcal{F}_h^{o}$, and $\mathcal{F}_h^{\partial}$ denote the set of faces, interior faces, and boundary faces, respectively. The NCRT element space in $\mathbb{R}^{d}$ is
\begin{equation*}
\begin{aligned}
\mathcal{V}_{g,h}^{(d)}:=&\{v\in L_2(\Omega): v|_{K}\in\text{span}\{1,x_{1},\ldots,x_{d},x_{1}^{2}-x_{2}^{2},\ldots,x_{1}^{2}-x_{d}^{2}\}\\
&\text{ for\ all}\ K\in\Th, ~\fint_{F}v\text{ is\ single-valued\ for\ all}\ F\in\mathcal{F}_h^{o}, \\
&\fint_{F}v=\fint_{F}g\text{ at the centroid of each } F\in\mathcal{F}_h^{\partial}\},\\
\end{aligned}
\end{equation*}
where $\fint_{F}v:=\frac{1}{|F|}\int_Fv$ is the surface mean of $v$ on $F$.  
The NCRT method for \eqref{elliptic} in $\mathbb{R}^{d}$ is to find $u_{h}^{(d)}\in\mathcal{V}_{g,h}^{(d)}$, such that
\begin{equation}\label{RT3}
\ab{a\nabla_{h}u_{h}^{(d)}}{\nabla_{h}v}+\ab{\bm{b}\cdot\nabla_{h}u_{h}^{(d)}}{v}+\langle cu_{h}^{(d)},v\rangle=\ab{f}{v},\quad\forall v\in\mathcal{V}_{0,h}^{(d)}.
\end{equation}
Let $Q_{1}^{(j)}(K)$ {\color{red}be} the space of polynomials on $K$ that are linear in $x_{j}$ and constant in $x_{i}$ for $i\neq j$. Let
\begin{equation*}
\mathcal{RT}_{h}^{(d)}:=\{\bm{\tau}_{h}\in H(\divg,\Omega): \bm{\tau}_{h}|_{K}\in \Pi_{j=1}^{d}Q_{1}^{(j)}(K)\text{ for all }K\in\Th\}.
\end{equation*}
The $H(\divg)$-space in $\mathbb{R}^{d}$ is $H(\divg;\Omega)=\{\bm{\tau}\in\Pi_{j=1}^{d}L_2(\Omega): \nabla\cdot\bm{\tau}\in L_2(\Omega)\}$. The next lemma is a direct genearlization of Lemma \ref{mainlemma}. The proof follows from direct (but tedious) calculation.
\begin{lemma}\label{lemma2}
Let $\bar{f}$ be a piecewise constant, $\bm{\tau}_{h}|_{K}\in\Pi_{j=1}^{d}Q_{1}^{(j)}(K)$ and $\nabla\cdot(\bm{\tau}_{h}|_{K})=0$ for all $K\in\Th$. Assume that
\begin{equation*}
\ab{\bm{\tau}_{h}}{\nabla_{h}v}=\ab{\bar{f}}{v}
\end{equation*}
for all $v\in\mathcal{V}_{0,h}^{(d)}$. Then $\bm{\tau}_{h}-\bar{f}\bm{r}_{h}^{(d)}\in\mathcal{RT}_{h}^{(d)},$ 
with
\begin{align*}
    &\bm{r}_{h}^{(d)}|_{K}(x_1,x_2,\ldots,x_d)\cdot\bm{e}_{i}\\
    &\quad:=\ell_{K,1}^{2}\ldots\widehat{\ell_{K,i}^{2}}\ldots\ell_{K,d}^{2}(x_{i}-x_{K,i})/\sum_{j=1}^{d}\ell_{K,1}^{2}\ldots\widehat{\ell_{K,j}^{2}}\ldots\ell_{K,d}^{2}
\end{align*}
for $1\leq i\leq d$,
where $\bm{e}_{i}$ is the $i$-th unit vector, $\widehat{\cdot}$ means the variable below is suppressed,  $K=\Pi_{j=1}^{d}[x_{j,i_j},x_{j,i_j+1}]$, $\ell_{K,j}=x_{j,i_j+1}-x_{j,i_j}$, and $(x_{K,1},\ldots,x_{K,d})$ is the centroid of $K$.
\end{lemma}

Given $\bm{\tau}\in\Pi_{j=1}^dH^1(\Omega)$, the $d$-dimensional RT interpolant $\Pi_{h}^{(d)}\bm{\tau}\in\mathcal{RT}^{(d)}_{h}$ is determined by 
\begin{equation}\label{RTinterpolation3}
    \int_F(\Pi^{(d)}_{h}\bm{\tau})\cdot\bm{n}_F=\int_F\bm{\tau}\cdot\bm{n}_F,\quad\forall F\in\mathcal{F}_h,
\end{equation}
where $\bm{n}_F$ is a unit normal to $F$.
By Lemma \ref{lemma2} and following exactly the same procedure in Section \ref{sec3}, we obtain a supercloseness estimate in $\mathbb{R}^{d}$.
\begin{theorem}\label{superclose3d} 
Let $Q_{h}^{(d)}$ be the $L_2$-projection onto $\nabla_{h}\mathcal{V}_{0,h}^{(d)}$ and 
\begin{equation*}
    \bm{\sigma}_{h}^{(d)}:=Q_{h}^{(d)}(a\nabla_{h}u_{h}^{(d)})-\bm{r}_{h}^{(d)}P_{h}(f-cu_{h}^{(d)}-\bm{b}\cdot\nabla_{h}u_{h}^{(d)}).
\end{equation*}
It holds that
\begin{equation*}
\|\Pi_{h}^{(d)}(a\nabla u)-\bm{\sigma}_{h}^{(d)}\|\lesssim
h^{2}\|u\|_{H^{3}}.
\end{equation*}
\end{theorem}

In particular, for $d=3$, we have
\begin{equation*}
\bm{r}_{h}^{(3)}|_{K}(\bm{x})=\frac{
\big(\ell_{K,2}^{2}\ell_{K,3}^{2}(x_{1}-x_{K,1}),\ell_{K,3}^{2}\ell_{K,1}^{2}(x_{2}-x_{K,2}),\ell_{K,1}^{2}\ell_{K,2}^{2}(x_{3}-x_{K,3})\big)^{T}}{\ell_{K,1}^{2}\ell_{K,2}^{2}+\ell_{K,2}^{2}\ell_{K,3}^{2}+\ell_{K,3}^{2}\ell_{K,1}^{2}}.
\end{equation*}

Let $A_{h}^{(3)}$ be the face-based weighed averaging generalized from $A_{h}$ in Definition \ref{defAh}. Using an argument very similar to the proof of Theorem \ref{superapprox}, one could show that $A_{h}^{(3)}\Pi_{h}^{(3)}\bm{\sigma}$ superconverges to $\bm{\sigma}$ in the $L_2$-norm. Hence we obtain the superconvergent flux recovery in $\mathbb{R}^{3}$.
\begin{theorem}\label{superconvergence3d}
For $d=3$, it holds that
\begin{equation*}
\|a\nabla u-A_{h}^{(3)}\bm{\sigma}_{h}^{(3)}\|\lesssim h^{2}\|u\|_{H^{3}}.
\end{equation*}
\end{theorem}
\begin{proof}
The proof is same as Theorems \ref{superapprox} and \ref{superconvergence}. We require $d=3$ since the inequality \eqref{l2max} with $h^{2-\frac{d}{2}}$ replacing $h$ does not hold for $d>3$.\qed
\end{proof}


\section{Numerical experiments}\label{sec5}
\begin{table}[tbhp]
\caption{Rate of convergence in $\mathbb{R}^{2}$}
\label{table2d}
\centering
\begin{tabular}{|c|c|c|c|c|}
\hline
ne & $\|u-u_{h}\|$
 &$\|a\nabla u-a\nabla_{h}u_{h}\|$
&  $\|\Pi_{h}(a\nabla u)-\tilde{\bm{\sigma}}_{h}\|$
&  $\|a\nabla u-A_{h}\tilde{\bm{\sigma}}_{h}\|$ \\
\hline
6 &3.455e-02&1.157e+00&5.551e-01&1.451e+00\\
24 &8.394e-03&5.723e-01&1.366e-01&4.591e-01\\
96  &2.112e-03&2.890e-01&3.509e-02&6.692e-02\\
384 &5.350e-04&1.457e-01&8.812e-03&1.274e-02\\
1536 &1.352e-04&7.316e-02&2.227e-03&2.969e-03\\
6144 &3.410e-05&3.671e-02&5.638e-04&7.318e-04\\
24576 &8.582e-06&1.841e-02&1.419e-04&1.826e-04\\ 
\hline
             order &2.045&1.023&2.042&2.098\\
\hline
\end{tabular}
\end{table}
In this section, we test the recovery operators $A_{h}$ and $A_{h}^{(3)}$. Instead of using $\bm{\sigma}_{h}$ analyzed in Sections \ref{sec3} and \ref{sec4}, we compute the modified flux $\tilde{\bm{\sigma}}_{h}$ in \eqref{sigmatilde} in the 2d experiment. For the numerical example in $\mathbb{R}^3$, we modify $\bm{\sigma}_{h}^{(3)}$ in Theorem \ref{superclose3d} and compute the flux $\tilde{\bm{\sigma}}_{h}^{(3)}$ given by 
\begin{equation}\label{sigmatilde3}
    \tilde{\bm{\sigma}}_{h}^{(3)}|_K=Q_{h}^{(3)}(a\nabla_{h}u_h^{(3)})-\bm{r}_{h}^{(3)}(f-cu_{h}^{(3)}-\bm{b}\cdot\nabla_{h}u_{h}^{(3)})(\bm{x}_K)
\end{equation}
on each cube $K\in\mathcal{T}_h,$
where $\bm{x}_K$ is the centroid of $K$. It is noted that $\nabla _{h}\mathcal{V}_{0,h}$ and $\nabla_{h}\mathcal{V}^{(3)}_{0,h}$ are broken spaces without any inter-element continuity. As a consequence, the projection $Q_{h}$ onto $\nabla _{h}\mathcal{V}_{0,h}$ in \eqref{sigmatilde} and the projection $Q^{(3)}_{h}$  onto $\nabla_{h}\mathcal{V}^{(3)}_{0,h}$ in \eqref{sigmatilde3} can be computed element-wise.  Based on Definition \ref{defAh},   the value of  $A_h\tilde{\bm{\sigma}}_h\in\widetilde{\mathcal{V}}_h$ at the midpoint of each interior edge is determined by a special weighted average of $\tilde{\bm{\sigma}}_h$ across that edge, while an extrapolation is used to compute $A_h\tilde{\bm{\sigma}}_h$ at midpoints of boundary edges. Recall that midpoint function values at all edges form the dofs of $\widetilde{\mathcal{V}}_h$ and correspond to locally supported basis functions of  $\widetilde{\mathcal{V}}_h$. Therefore one could combine midpoint values of $A_h\tilde{\bm{\sigma}}_h$ and the induced basis of  $\widetilde{\mathcal{V}}_h$ to compute the value of $A_h\tilde{\bm{\sigma}}_h$ at any necessary discrete points. The postprocessed flux $A^{(3)}_h\tilde{\bm{\sigma}}^{(3)}_h$ in $\mathbb{R}^3$ is calculated in a similar way. 

To compute the RT interpolant $\Pi_h(a\nabla u)$, it suffices to use RT edge basis functions and the dof $\int_E(a\nabla u)\cdot\bm{n}_E$ on each edge $E\in\mathcal{E}_h,$ see \eqref{RTinterpolation}. The 4-point Gaussian quadrature $\{(b_i,c_i)\}_{i=1}^4$ is used to approximate the edge integral $\int_E(a\nabla u)\cdot\bm{n}_E$, where $\{b_i\}_{i=1}^4$ are positive weights and $\{c_i\}_{i=1}^4$ are coordinates of quadrature points on a reference interval. As for the interpolant $\Pi^{(3)}_h(a\nabla u)$ in $\mathbb{R}^3$, the related face integral $\int_F(a\nabla u)\cdot\bm{n}_F$ (see \eqref{RTinterpolation3}) is evaluated using the 2d \emph{tensor product} of $\{(b_i,c_i)\}_{i=1}^4$ with 16 interior quadrature points on each rectangular face $F$. When assembling stiffness matrices and right hand sides, we use the 2d (resp.~3d) tensor product of $\{(b_i,c_i)\}_{i=1}^4$ to approximate integrals on rectangular (resp.~cubical) elements. The 3d quadrature rule in each cube makes use of $4^3=64$ quadrature points.

The basis of $\mathcal{V}_{0,h}$ (resp.~$\mathcal{V}^{(3)}_{0,h}$) is chosen to be dual to the dofs $\{\fint_E\cdot\}_{E\in\mathcal{E}^o_h}$ (resp.~$\{\fint_F\cdot\}_{F\in\mathcal{F}^o_h}$). With such a basis and the aforementioned element-wise approximate integration, we could numerically solve \eqref{RT} (resp.~\eqref{RT3}) to obtain the dofs of $u_h$ (resp.~$u_h^{(3)}$). Those dofs are then combined with the dual basis to calculate
$u_h$ and $u_h^{(3)}$ at the discrete quadrature points necessary for integral quantities shown in Tables \ref{table2d} and \ref{table3d}.

In each table, `ne' denotes the number of elements in $\Th$. The order of convergence is $p$ such that the error $\approx Ch^{p}$ with some constant $C$ independent of $h$. We evaluate $p$ by least squares using the data in Tables \ref{table2d} and \ref{table3d}. 

\textbf{Problem 1:}
Consider the equation \eqref{elliptic} with $\Omega=[0,1]\times[0,1]$, 
\begin{equation*}
\begin{aligned}
&u=\exp(2x_{1}+x_{2})x_{1}^{2}(x_{1}-1)^{2}x_{2}^{2}(x_{2}-1)^{2},\\
&a(\bm{x})=\exp(x_{1}),\quad\bm{b}(\bm{x})=\bm{x},\quad c(\bm{x})=\exp(x_{1}+x_{2}),
\end{aligned}
\end{equation*}
and corresponding $g$ and $f$. The initial rectangular mesh is
$$\Th=\bigcup_{0\leq i\leq2,0\leq j\leq1}[x_{1,i},x_{1,i+1}]\times[x_{2,j},x_{2,j+1}],$$
where $x_{1,0}=0, x_{1,1}=0.4, x_{1,2}=0.8, x_{1,3}=1$ and $x_{2,0}=0, x_{2,1}=0.7, x_{2,2}=1$. We refine the mesh by connecting the midpoints of opposite edges of each rectangle. In the refinement, we randomly perturb the mesh along $x_{1}$- and $x_{2}$-directions by $20\%$ of the length of the smallest interval in that direction, respectively. Numerical results are presented in Table \ref{table2d}.  The first three rows in Table \ref{table2d} are not used to evaluate the order since they are outside of the asymptotic regime. 

\begin{table}[tbhp]
\caption{Rate of convergence in $\mathbb{R}^{3}$}
\label{table3d}
\centering
\begin{tabular}{|c|c|c|c|c|}
\hline
ne & $\|u-u_{h}^{(3)}\|$
 &$\|a\nabla u-a\nabla_{h}u_{h}^{(3)}\|$
&  $\|\Pi_{h}^{(3)}(a\nabla u)-\tilde{\bm{\sigma}}_{h}^{(3)}\|$
&  $\|a\nabla u-A_{h}^{(3)}\tilde{\bm{\sigma}}_{h}^{(3)}\|$ \\
\hline
8&9.341e-01&1.280e+01&1.863e+01 	& 2.238e+01 \\
64&4.158e-01&9.418e+00 	& 5.547e+00 &	 1.516e+01 \\
512&1.200e-01 	& 5.032e+00 	& 1.902e+00 &	 3.448e+00 \\
4096&3.010e-02 &	 2.525e+00 &	 4.967e-01& 	 8.599e-01 \\
32768&7.661e-03 	& 1.269e+00 	& 1.285e-01 &	 1.709e-01 \\
\hline
             order &2.085&1.044&2.042&2.274\\
\hline
\end{tabular}
\end{table}

\textbf{Problem 2:}
In the second experiment, we consider the equation \eqref{elliptic} with $\Omega=[0,1]\times[0,1]\times[0,1]$, 
\begin{equation*}
\begin{aligned}
&u(\bm{x})=\exp(x_{1}+x_{2})\sin(3\pi x_{1})\sin(2\pi x_{2})\sin(\pi x_{3}),\\
&a(\bm{x})=\exp(x_{1}+x_{2}+x_{3}),\quad\bm{b}(\bm{x})=\bm{0},\quad c(\bm{x})=0,
\end{aligned}
\end{equation*}
and corresponding $g$ and $f$. The initial cubical mesh is
$$\Th=\bigcup_{0\leq i\leq1,0\leq j\leq1,0\leq k\leq1}[x_{1,i},x_{1,i+1}]\times[x_{2,j},x_{2,j+1}]\times[x_{3,k},x_{3,k+1}],$$
where 
\begin{align*}
    &(x_{1,0},x_{1,1},x_{1,2})=(0,0.5,1),\\
    &(x_{2,0},x_{2,1},x_{2,2})=(0,0.6,1),\\ &(x_{3,0},x_{3,1},x_{3,2})=(0,0.4,1).
\end{align*}
We refine the mesh by connecting the centroid of opposite faces of each element. In the refinement, we randomly perturb the mesh along $x_{1}$-, $x_{2}$-, and $x_{3}$-directions by $20\%$ of the length of the smallest interval in that direction, respectively. Numerical results are presented in Table \ref{table3d}.  For similar reason, the first two rows are not used. 

In the two experiments, since the mesh is randomly perturbed, computed errors are not exactly the same (but similar) every time. The numerical results show that our superconvergence estimates Theorems \ref{superclose}, \ref{superconvergence}, and \ref{superclose3d} are asymptotically sharp. We also note that the rate of convergence in the last column of Table \ref{table3d} is slightly larger than the predicted order $2$ from Theorem \ref{superconvergence3d}. One possible reason is that the mesh size in $\mathbb{R}^3$ is not small enough. In fact, the numerical solution of \eqref{RT3} on the uniform refinement of the finest mesh in Table \ref{table3d} is beyond the computational power of our machine.

\section{Concluding remarks}\label{sec6}
We have developed a superconvergent flux recovery process for NCRT and CR element methods for second order elliptic equations. It is well-known that these elements are originally designed for efficiently solving the Stokes equation, see \cite{CR1973,RT1992}. Hence, extending our analysis and results to the Stokes equation is of  practical interest and a direction of future research. 

\section{Declarations}

\textbf{Funding} The author did not receive support from any organization for this work.

\vspace{0.2cm}
\noindent\textbf{Conflicts of interest} The author has no relevant financial or non-financial interests to disclose. 

\vspace{0.2cm}
\noindent\textbf{Availability of data} Data sharing is not applicable to this article as no datasets were generated or analysed during the current study.

\vspace{0.2cm}
\noindent\textbf{Code availability} The code used in this study is available from the author upon request.





\begin{thebibliography}{10}
\providecommand{\url}[1]{{#1}}
\providecommand{\urlprefix}{URL }
\expandafter\ifx\csname urlstyle\endcsname\relax
  \providecommand{\doi}[1]{DOI~\discretionary{}{}{}#1}\else
  \providecommand{\doi}{DOI~\discretionary{}{}{}\begingroup
  \urlstyle{rm}\Url}\fi

\bibitem{AC1995}
Arbogast, T., Chen, Z.: On the implementation of mixed methods as nonconforming
  methods for second-order elliptic problems.
\newblock Math. Comp. \textbf{64}(211), 943--972 (1995)

\bibitem{AB1985}
Arnold, D.N., Brezzi, F.: Mixed and nonconforming finite element methods:
  implementation, postprocessing and error estimates.
\newblock RAIRO Mod\'{e}l. Math. Anal. Num\'{e}r. \textbf{19}(1), 7--32 (1985).
\newblock \doi{10.1051/m2an/1985190100071}.
\newblock \urlprefix\url{https://doi.org/10.1051/m2an/1985190100071}

\bibitem{BaLi2019}
Bank, R.E., Li, Y.: Superconvergent recovery of {R}aviart-{T}homas mixed finite
  elements on triangular grids.
\newblock J. Sci. Comput. \textbf{81}(3), 1882--1905 (2019).
\newblock \doi{10.1007/s10915-019-01068-0}.
\newblock \urlprefix\url{https://doi.org/10.1007/s10915-019-01068-0}

\bibitem{BX2003a}
Bank, R.E., Xu, J.: Asymptotically exact a posteriori error estimators. {I}.
  {G}rids with superconvergence.
\newblock SIAM J. Numer. Anal. \textbf{41}(6), 2294--2312 (2003).
\newblock \doi{10.1137/S003614290139874X}.
\newblock \urlprefix\url{https://doi.org/10.1137/S003614290139874X}

\bibitem{BX2003b}
Bank, R.E., Xu, J.: Asymptotically exact a posteriori error estimators. {II}.
  {G}eneral unstructured grids.
\newblock SIAM J. Numer. Anal. \textbf{41}(6), 2313--2332 (2003).
\newblock \doi{10.1137/S0036142901398751}.
\newblock \urlprefix\url{https://doi.org/10.1137/S0036142901398751}

\bibitem{BaXuZheng2007}
Bank, R.E., Xu, J., Zheng, B.: Superconvergent derivative recovery for lagrange
  triangular elements of degree p on unstructured grids.
\newblock SIAM J. Numer. Anal. \textbf{45}(5), 2032--2046 (2007)

\bibitem{BS1977}
Bramble, J.H., Schatz, A.H.: Higher order local accuracy by averaging in the
  finite element method.
\newblock Math. Comp. \textbf{31}(137), 94--111 (1977)

\bibitem{Brandts1994}
Brandts, J.H.: Superconvergence and a posteriori error estimation for
  triangular mixed finite elements.
\newblock Numer. Math. \textbf{68}(3), 311--324 (1994).
\newblock \doi{10.1007/s002110050064}.
\newblock \urlprefix\url{https://doi.org/10.1007/s002110050064}

\bibitem{BS2008}
Brenner, S.C., Scott, L.R.: The mathematical theory of finite element methods,
  \emph{Texts in Applied Mathematics, 15}, vol.~35, 3 edn.
\newblock Springer, New York (2008)

\bibitem{CaoHuang2017}
Cao, W., Huang, Q.: Superconvergence of local discontinuous {G}alerkin methods
  for partial differential equations with higher order derivatives.
\newblock J. Sci. Comput. \textbf{72}(2), 761--791 (2017).
\newblock \doi{10.1007/s10915-017-0377-z}.
\newblock \urlprefix\url{https://doi.org/10.1007/s10915-017-0377-z}

\bibitem{CSYZ2015}
Cao, W., Shu, C.W., Yang, Y., Zhang, Z.: Superconvergence of discontinuous
  {G}alerkin methods for two-dimensional hyperbolic equations.
\newblock SIAM J. Numer. Anal. \textbf{53}(4), 1651--1671 (2015).
\newblock \doi{10.1137/140996203}.
\newblock \urlprefix\url{https://doi.org/10.1137/140996203}

\bibitem{CSYZ2018}
Cao, W., Shu, C.W., Yang, Y., Zhang, Z.: Superconvergence of discontinuous
  {G}alerkin method for scalar nonlinear hyperbolic equations.
\newblock SIAM J. Numer. Anal. \textbf{56}(2), 732--765 (2018).
\newblock \doi{10.1137/17M1128605}.
\newblock \urlprefix\url{https://doi.org/10.1137/17M1128605}

\bibitem{ChenHu2013}
Chen, C., Hu, S.: The highest order superconvergence for bi-{$k$} degree
  rectangular elements at nodes: a proof of {$2k$}-conjecture.
\newblock Math. Comp. \textbf{82}(283), 1337--1355 (2013).
\newblock \doi{10.1090/S0025-5718-2012-02653-6}.
\newblock \urlprefix\url{https://doi.org/10.1090/S0025-5718-2012-02653-6}

\bibitem{Chen2002}
Chen, C.M.: Structure theory of superconvergence of finite elements (in
  Chinese).
\newblock Hunan Science and Technology Press, Changsha (2002)

\bibitem{ChenLi1994}
Chen, H., Li, B.: Superconvergence analysis and error expansion for the
  {W}ilson nonconforming finite element.
\newblock Numer. Math. \textbf{69}(2), 120--140 (1994)

\bibitem{CoGuWa2009}
Cockburn, B., Guzm\'{a}n, J., Wang, H.: Superconvergent discontinuous
  {G}alerkin methods for second-order elliptic problems.
\newblock Math. Comp. \textbf{78}(265), 1--24 (2009).
\newblock \doi{10.1090/S0025-5718-08-02146-7}.
\newblock \urlprefix\url{https://doi.org/10.1090/S0025-5718-08-02146-7}

\bibitem{CR1973}
Crouzeix, M., Raviart, P.A.: Conforming and nonconforming finite element
  methods for solving the stationary {S}tokes equations.
\newblock RAIRO Anal. Num\'er. \textbf{7}(R-3), 33--75 (1973)

\bibitem{RyanShu2007}
Curtis, S., Kirby, R.M., Ryan, J.K., Shu, C.W.: Postprocessing for the
  discontinuous {G}alerkin method over nonuniform meshes.
\newblock SIAM J. Sci. Comput. \textbf{30}(1), 272--289 (2007/08).
\newblock \doi{10.1137/070681284}.
\newblock \urlprefix\url{https://doi.org/10.1137/070681284}

\bibitem{DuZhang2019}
Du, Y., Zhang, Z.: Supercloseness of linear {DG}-{FEM} and its superconvergence
  based on the polynomial preserving recovery for {H}elmholtz equation.
\newblock J. Sci. Comput. \textbf{79}(3), 1713--1736 (2019).
\newblock \doi{10.1007/s10915-019-00906-5}.
\newblock \urlprefix\url{https://doi.org/10.1007/s10915-019-00906-5}

\bibitem{Duran1990}
Dur\'{a}n, R.: Superconvergence for rectangular mixed finite elements.
\newblock Numer. Math. \textbf{58}(3), 287--298 (1990).
\newblock \doi{10.1007/BF01385626}.
\newblock \urlprefix\url{https://doi.org/10.1007/BF01385626}

\bibitem{Duran1991}
Dur\'an, R., Muschietti, M.A., Rodriguez, R.: On the asymptotic exactness of
  error estimators for linear triangular finite elements.
\newblock Numer. Math. \textbf{59}(2), 107--127 (1991)

\bibitem{GuoZhang2015}
Guo, H., Zhang, Z.: Gradient recovery for the {C}rouzeix-{R}aviart element.
\newblock J. Sci. Comput. \textbf{64}(2), 456--476 (2015)

\bibitem{HuMa2018}
{Hu}, J., {Ma}, L., {Ma}, R.: {Optimal Superconvergence Analysis for the
  Crouzeix-Raviart and the Morley elements}.
\newblock arXiv e-prints arXiv:1808.09810 (2018)

\bibitem{HM2016}
Hu, J., Ma, R.: Superconvergence of both the {C}rouzeix-{R}aviart and morley
  elements.
\newblock Numer. Math. \textbf{132}(3), 491--509 (2016)

\bibitem{LMW2000}
Lakhany, A.M., Marek, I., Whiteman, J.R.: Superconvergence results on mildly
  structured triangulations.
\newblock Comput. Methods Appl. Mech. Engrg. \textbf{189}, 1--75 (2000)

\bibitem{YL2018}
Li, Y.W.: Global superconvergence of the lowest-order mixed finite element on
  mildly structured meshes.
\newblock SIAM J. Numer. Anal. \textbf{56}(2), 792--815 (2018).
\newblock \doi{10.1137/17M112587X}.
\newblock \urlprefix\url{https://doi.org/10.1137/17M112587X}

\bibitem{LTZ2005}
Lin, Q., Tobiska, L., Zhou, A.: Superconvergence and extrapolation of
  non-conforming low order finite elements applied to the poisson equation.
\newblock IMA J. Numer. Anal. \textbf{25}(1), 160--181 (2005)

\bibitem{LN2008}
Liu, H., Yan, N.: Superconvergence analysis of the nonconforming quadrilateral
  linear-constant scheme for {S}tokes equations.
\newblock Adv. Comput. Math. \textbf{29}(4), 375--392 (2008).
\newblock \doi{10.1007/s10444-007-9054-3}.
\newblock \urlprefix\url{https://doi.org/10.1007/s10444-007-9054-3}

\bibitem{MS2009}
Mao, S., Shi, Z.C.: High accuracy analysis of two nonconforming plate elements.
\newblock Numer. Math. \textbf{111}(3), 407--443 (2009)

\bibitem{Marini1985}
Marini, L.D.: An inexpensive method for the evaluation of the solution of the
  lowest order raviart-thomas mixed method.
\newblock SIAM J. Numer. Anal. \textbf{22}(3), 493--496 (1985)

\bibitem{MSX2006}
Ming, P., Shi, Z.C., Xu, Y.: Superconvergence studies of quadrilateral
  nonconforming rotated {Q}1 elements.
\newblock Int. J. Numer. Anal. Model. \textbf{3}(3), 322--332 (2006)

\bibitem{RT1992}
Rannacher, R., Turek, S.: Simple nonconforming quadrilateral {S}tokes element.
\newblock Numer. Methods Partial Differential Equations \textbf{8}(2), 97--111
  (1992)

\bibitem{Thomee1977}
Thom\'{e}e, V.: High order local approximations to derivatives in the finite
  element method.
\newblock Math. Comp. \textbf{31}(139), 652--660 (1977).
\newblock \doi{10.2307/2005998}.
\newblock \urlprefix\url{https://doi.org/10.2307/2005998}

\bibitem{Wang2000}
Wang, J.: Superconvergence analysis for finite element solutions by the
  least-squares surface fitting on irregular meshes for smooth problems.
\newblock J. Math. Study \textbf{33}(3), 229--243 (2000)

\bibitem{WangYe2001}
Wang, J., Ye, X.: Superconvergence of finite element approximations for the
  {S}tokes problem by projection methods.
\newblock SIAM J. Numer. Anal. \textbf{39}(3), 1001--1013 (2001).
\newblock \doi{10.1137/S003614290037589X}.
\newblock \urlprefix\url{https://doi.org/10.1137/S003614290037589X}

\bibitem{XZ2003}
Xu, J., Zhang, Z.: Analysis of recovery type a posteriori error estimators for
  mildly structured grids.
\newblock Math. Comp. \textbf{73}(247), 1139--1152 (2004).
\newblock \doi{10.1090/S0025-5718-03-01600-4}.
\newblock \urlprefix\url{https://doi.org/10.1090/S0025-5718-03-01600-4}

\bibitem{Ye2001}
Ye, X.: Superconvergence of nonconforming finite element method for the
  {S}tokes equations.
\newblock Numer. Methods Partial Differential Equations \textbf{18}(2),
  143--154 (2002).
\newblock \doi{10.1002/num.1036.abs}.
\newblock \urlprefix\url{https://doi.org/10.1002/num.1036.abs}

\bibitem{ZHY2019}
Zhang, Y., Huang, Y., Yi, N.: Superconvergence of the {C}rouzeix-{R}aviart
  element for elliptic equation.
\newblock Adv. Comput. Math.  (2019).
\newblock
  \href{https://doi.org/10.1007/s10444-019-09714-9}{doi:10.1007/s10444-019-09714-9}

\bibitem{ZhangNaga2005}
Zhang, Z., Naga, A.: A new finite element gradient recovery method:
  superconvergence property.
\newblock SIAM J. Sci. Comput. \textbf{26}(4), 1192--1213 (2005).
\newblock \doi{10.1137/S1064827503402837}.
\newblock \urlprefix\url{https://doi.org/10.1137/S1064827503402837}

\bibitem{ZZ1987}
Zienkiewicz, O.C., Zhu, J.Z.: A simple error estimator and adaptive procedure
  for practical engineering analysis.
\newblock Internat. J. Numer. Methods Engrg. \textbf{24}(2), 337--357 (1987).
\newblock \doi{10.1002/nme.1620240206}.
\newblock \urlprefix\url{https://doi.org/10.1002/nme.1620240206}

\bibitem{ZZ1992}
Zienkiewicz, O.C., Zhu, J.Z.: The superconvergent patch recovery and a
  posteriori error estimates. {I}. {T}he recovery technique.
\newblock Internat. J. Numer. Methods Engrg. \textbf{33}(7), 1331--1364 (1992).
\newblock \doi{10.1002/nme.1620330702}.
\newblock \urlprefix\url{https://doi.org/10.1002/nme.1620330702}

\end{thebibliography}
\end{document}